\documentclass[reqno, onesiden 12pt]{amsart}

\usepackage[letterpaper]{geometry}
\usepackage{enumerate, hyperref,url,amssymb,amsmath,amsthm,amsxtra,mathtools,mathrsfs,calc,nccmath,color}

\usepackage{calc}
\usepackage{graphicx}
\usepackage{caption}
\usepackage{subcaption}

\newcommand{\Z}{\mathbb{Z}}
\newcommand{\Q}{\mathbb{Q}}

\newcommand{\C}{\mathbb{C}}

\let\temp\phi
\let\phi\varphi
\let\varphi\temp

\renewcommand{\(}{\left(}
\renewcommand{\)}{\right)}

\newcommand{\SL}{\operatorname{SL}}
\newcommand{\GL}{\operatorname{GL}}

\newcommand{\Sh}{\operatorname{Sh}}
\newcommand{\Gal}{\operatorname{Gal}}

\renewcommand{\sl}{\big|}
\newcommand{\sk}{\big|_k }

\renewcommand{\bar}[1]{\overline{#1}}

\newtheorem{theorem}{Theorem}[section]
\newtheorem{lemma}[theorem]{Lemma}

\newtheorem{proposition}[theorem]{Proposition}

\theoremstyle{remark}
\newtheorem*{remark}{Remark}

\numberwithin{equation}{section}

\begin{document}


\title{Congruences for Level $1$ cusp forms of half-integral weight}

\date{\today}
\author{Robert Dicks}
\address{Department of Mathematics\\
University of Illinois\\
Urbana, IL 61801} 
\email{rdicks2@illinois.edu}

 
\begin{abstract}
Suppose that $\ell \geq 5$ is prime. 
For a positive integer $N$ with $4 \mid N$, previous works studied properties of half-integral weight modular forms on $\Gamma_0(N)$ which are supported on finitely many square classes modulo $\ell$, in some cases proving that these forms are congruent to the image of a single variable theta series under some number of iterations of the Ramanujan $\Theta$-operator. Here, we study the analogous problem for modular forms of half-integral weight on $\SL_{2}(\Z)$. Let $\eta$ be the Dedekind eta function. For a wide range of weights, we prove that every half-integral weight modular form on $\SL_{2}(\Z)$ which is supported on finitely many square classes modulo $\ell$ can be written modulo $\ell$ in terms of $\eta^{\ell}$ and an iterated derivative of $\eta$.
\end{abstract}


\maketitle

 
 \section{Introduction}
Suppose that $\ell \geq 5$ is prime. 
Many papers \cite{Ono-Skinner} \cite{Bruinier} \cite{Bruinier-Ono} \cite{Ahlgren-Boylan} \cite{Ahlgren-Boylan2} \cite{Ahlgren-Choi-Rouse} \cite{Ahlgren-Rouse}
study half-integral weight modular forms with few non-vanishing coefficients modulo $\ell$ and give applications for divisibility properties of the algebraic parts of the central critical values of modular $L$-functions and the orders of Tate-Shafarevich groups of elliptic curves.

These results are modulo $\ell$ analogues of a theorem of Vign\`{e}ras in characteristic $0$. 
If $\lambda$ is a non-negative integer and $N$ is a positive integer with $4\mid N$, let $M_{\lambda+\frac{1}{2}}(\Gamma_1(N))$ be the space of modular forms of weight $\lambda+\frac{1}{2}$ (in the sense of \cite{Shimura}) on $\Gamma_1(N)$. 
Vign\`{e}ras 
proved that a form $F(z) \in M_{\lambda+\frac{1}{2}}(\Gamma_1(N))$ whose coefficients are supported on finitely many square classes of integers is a linear combination of single-variable theta series. The precise result is below (Bruinier \cite{Bruinier-Vigneras} gave a different proof of this theorem ).

 \begin{theorem}{\cite{Vigneras}}\label{thm:Vigneras}
Suppose that $\lambda \geq 0$ is an integer, that $N$ is a positive integer with $4 \mid N$, and that $F(z) \in M_{\lambda+ \frac{1}{2}}(\Gamma_1(N))$. If there exist finitely many square-free integers $t_1$,$t_2$,...,$t_m$ for which

\[
F(z)=\sum_{i=1}^{m}\sum_{n=0}^{\infty}a(t_{i}n^{2})q^{t_{i}n^{2}} , \ \ \ q= e^{2\pi i z}
\]
 then $\lambda=0 \text{ or } 1$ and $F(z)$ is a linear combination of theta series.
 \end{theorem}
 
 A recent result of Bella\"{\i}che, Green and Soundararajan \cite{Belliache-Green-Sound} implies for any half-integral weight modular form that the number of coefficients $\leq X$ which do not vanish modulo $\ell$
is $\gg\frac{\sqrt X}{\log\log X}$.  It is natural to suspect that the only half-integral weight forms for which the number of non-vanishing coefficients is close to this lower bound are
those which  are supported on finitely many square classes modulo $\ell$.  Forms of half-integral weight on  $\SL_2(\Z)$ whose coefficients are sparse modulo $\ell$ play an important role in the recent work of Ahlgren, Beckwith and Raum \cite{Scarcity} on scarcity of congruences for the partition function.  
 
 Ahlgren, Choi and Rouse 
proved a modulo $\ell$ analogue of Theorem~\ref{thm:Vigneras} for forms $f(z)$ in the Kohnen plus-space $S_{\lambda+\frac{1}{2}}^{+}(\Gamma_{0}(4))$. Their main theorem was the following.
 
 \begin{theorem}{\cite{Ahlgren-Choi-Rouse}}\label{Ahlgren-Choi-RouseMainTheorem}
 Suppose that $\ell \geq 5$ is prime and that $K$ is a number field. Fix an embedding of $K$ into $\C$ and a prime  $v$ of $K$ above $\ell$. Let $\mathcal{O}_{v}$ denote the ring of $v$-integral elements of $K$. Suppose that $f \in S^{+}_{\lambda+\frac{1}{2}}(\Gamma_0(4)) \cap \mathcal{O}_{v}[[q]]$ satisfies
 \[
 f \equiv \sum_{i=1}^{m}\sum_{n=1}^{\infty}a(t_{i}n^{2})q^{t_{i}n^{2}} \not \equiv 0 \pmod{v},
 \]
 where each $t_{i}$ is a square-free positive integer. If $\lambda+\frac{1}{2} < \ell(\ell+1+\frac{1}{2})$, then $\lambda$ is even and
 \[
 f \equiv a(1)\sum_{n=1}^{\infty}n^{\lambda}q^{n^{2}} \pmod{v}.
 \]
 \end{theorem}

 In this paper, we study the analogous question for half-integral weight modular forms on $\SL_{2}(\Z)$. Before we state our main result, we introduce some notation. If $\lambda \geq 0$ is an integer, $N$ is a positive integer, and $\nu$ is a multiplier system on $\Gamma_0(N)$ in weight $\lambda+\frac{1}{2}$, we denote by $S_{\lambda+\frac{1}{2}}(N, \nu)$ the space of cusp forms of weight $\lambda+\frac{1}{2}$ and multiplier $\nu$ on $\Gamma_0(N)$ (details will be given in the next section). Let $\nu_{\eta}$ be the multiplier for the Dedekind eta function defined in \eqref{etamultiplier}. With this notation, we prove the following theorem.

 \begin{theorem}\label{thm:main}
 Suppose that $\ell \geq 5$ is prime and that $K$ is a number field. Fix an embedding of $K$ into $\C$ and a prime  $v$ of $K$ above $\ell$. 
 Let $\mathcal{O}_{v}$ denote the ring of $v$-integral elements in $K$. Suppose that $\lambda$ is a non-negative integer satisfying $\lambda +\frac{1}{2} < \frac{\ell^{2}}{2}$. Suppose that $r$ is a positive integer with $(r,6)=1$ and that $f \in S_{\lambda+\frac{1}{2}}(1, \nu^{r}_{\eta}) \cap \mathcal{O}_{v}[[q^{\frac{r}{24}}]]$ satisfies
 
 \[
 f \equiv \sum_{i=1}^{m}\sum_{n=1}^{\infty}a(t_{i}n^2)q^{\frac{t_{i}n^{2}}{24}}  \not \equiv 0 \pmod v,
 \]
 where each $t_{i}$ is a square-free positive integer. Then one of the following is true.
 
 \begin{enumerate}
 \item 
 $f \equiv a(1)\displaystyle \sum_{n=1}^{\infty}\(\mfrac{12}{n}\)n^{\lambda}q^{\frac{n^2}{24}} \pmod v $.

In this case, $r \equiv 1 \pmod{24}$ and $\lambda$ is even.
 
  \item
  $f \equiv \displaystyle a(\ell)\sum_{n=1}^{\infty}\(\mfrac{12}{n}\)q^{\frac{\ell n^{2}}{24}} \pmod v $.
 
 In this case, 
  $r \equiv \ell \pmod{24}$ and $\lambda \equiv \frac{\ell-1}{2} \pmod{\ell-1}$.

 \item
  $f \equiv a(1)\displaystyle \sum_{n=1}^{\infty}\(\mfrac{12}{n}\)n^{\lambda}q^{\frac{n^2}{24}}+a(\ell)\sum_{n=1}^{\infty}\(\mfrac{12}{n}\)q^{\frac{\ell n^{2}}{24}} \pmod v $, 
  where $a(1) \not \equiv 0 \pmod{v}$ and $a(\ell) \not \equiv  0 \pmod{v}$.
In this case, $r \equiv \ell \equiv ~1 \pmod{24}$ and $\lambda \equiv \frac{\ell -1}{2} \pmod{\ell-1}$.  
 \end{enumerate}
 
 \end{theorem}

 \begin{remark}
 For an example of case $(1)$ of Theorem~\ref{thm:main}, let $\ell \geq 5$ be prime and $\lambda$ be a nonnegative integer. 
 Lemma~\ref{Lemma1} below implies that there exists a form $f \in S_{(\frac{\lambda}{2})(\ell+1)}(1,\nu_{\eta})$ such that $f\equiv\Theta^{\frac{\lambda}{2}}(\eta)\pmod{\ell}$, where $\Theta$ is the Ramanujan $\Theta$-operator defined in \eqref{RamanujanTheta}. We have
 \[
 f \equiv \sum_{n=1}^{\infty}\(\frac{12}{n}\)n^{\lambda}q^{\frac{n^{2}}{24}} \pmod{\ell}.
 \]
 For an example of case $(2)$ of Theorem~\ref{thm:main}, set $f=\eta^{\ell}$. Since $\eta^{\ell}(z) \equiv \eta(\ell z) \pmod{\ell} $, we have
 \[
 f \equiv \sum_{n=1}^{\infty}\(\frac{12}{n}\)q^{\frac{\ell n^{2}}{24}} \pmod{\ell}.
 \]
  For an example of case $(3)$, suppose that $\ell$ is a prime such that $\ell \equiv 1 \pmod{24}$. Lemma~\ref{Lemma1} implies that there exists a form $g \in S_{(\frac{\ell-1}{4})(\ell+1)+\frac{1}{2}}(1,\nu_{\eta})$ such that 
  $g \equiv \Theta^{\frac{\ell-1}{4}}(\eta) \pmod{\ell}$. Set $f=24^{\frac{\ell-1}{2}}g+\eta^{\ell}E_{\ell-1}^{\frac{\ell-1}{4}}$ . We have
  \[
  f \equiv \sum_{n=1}^{\infty}\(\frac{12}{n}\)n^{\frac{\ell-1}{2}}q^{\frac{n^{2}}{24}}+\sum_{n=1}^{\infty}\(\frac{12}{n}\)q^{\frac{\ell n^{2}}{24}} \pmod{\ell}.
  \]
  For this example, note that $\frac{\ell-1}{4}(\ell+1) \equiv \frac{\ell-1}{2} \pmod{\ell-1}$.
 \end{remark}
 \begin{remark}
 The upper bound on $\lambda$ is sharp. For an example which illustrates this, set $f=\eta^{\ell^{2}}$. Then
 \[
 f \equiv \sum_{n=1}^{\infty}\(\frac{12}{n}\)q^{\frac{\ell^{2}n^{2}}{24}} \pmod{\ell}.
 \]
 Note that we have $\lambda+\frac{1}{2}=\frac{\ell^{2}}{2}$ in this case.
\end{remark}
 The paper is organized as follows. In Section $2$, we give some background results on modular forms of integral and half-integral weight. In Section $3$, we prove some preliminary results. In Section $4$, we make a preliminary reduction for the proof of Theorem~\ref{thm:main}, and in Section $5$, we prove the theorem.
 \section{Background}
Suppose that $k \in \frac{1}{2}\mathbb{Z}$, that $N$ is a positive integer, and that $\chi$ is a Dirichlet character modulo $N$. For a function $f(z)$ on the upper half plane and 
\[
 \gamma =\left(\begin{matrix}a & b \\c & d\end{matrix}\right) \in \GL_{2}^{+}(\Q),
\]
we have the weight $k$ slash operator 
\[
f(z)\sk \gamma := \det(\gamma)^{\frac{k}{2}}(cz+d)^{-k}f\(\frac{az+b}{cz+d}\).
\]

Suppose that $\ell \geq 5$ is prime and that $K$ is a number field. Fix an embedding of $K$ into $\C$ and a prime  $v$ of $K$ above $\ell$.
Let $\mathcal{O}_{v}$ be the ring of $v$-integral elements of $K$.
If $\nu$ is a multiplier system on $\Gamma_0(N)$, 
we denote by $M_{k}(N,\nu)$, $S_{k}(N,\nu)$
and $M_{k}^{!}(N, \nu)$ 
the spaces of modular forms, cusp forms, 
and weakly holomorphic modular forms 
of weight $k$ and multiplier 
$\nu$ on $\Gamma_0(N)$ whose Fourier coefficients are in $\mathcal{O}_{v}$.
When $k$ is an integer and the multiplier $\nu$ is trivial, we write $M_{k}(N)$, $S_{k}(N)$
and $M_{k}^{!}(N)$.
Forms in these spaces satisfy the transformation law

\[
f \sk \gamma= \nu(\gamma)f \ \ \ \text{ for } \ \ \ \gamma = \left(\begin{matrix}a & b \\c & d\end{matrix}\right) \in \Gamma_0(N)
\]
and the appropriate conditions at the cusps of $\Gamma_0(N)$.

Throughout, let $q:=e(z)=e^{2 \pi i z}$. 
We define the eta function by
\[
\eta(z):= q^{\frac{1}{24}}\prod_{n=1}^{\infty}(1-q^{n})
\]
and the theta function by
\[
\theta(z):= \sum_{n=-\infty}^{\infty}q^{n^2}.
\]
The eta function has a multiplier $\nu_{\eta}$ satisfying

\[
\eta(\gamma z)=\nu_{\eta}(\gamma)(cz+d)^{\frac{1}{2}}\eta(z), \ \ \ \ \ \gamma= \left(\begin{matrix}a & b \\c & d\end{matrix}\right) \in \SL_{2}(\Z);
\]
throughout, we choose the principal branch of the square root. For $c>0$, we have the formula \cite[~$\mathsection$$4.1$]{Knopp}

\begin{equation}\label{etamultiplier}
\nu_{\eta}(\gamma)=
 \begin{cases} 
 \(\frac{d}{c}\)e\(\frac{1}{24}((a+d)c-bd(c^2-1)-3c )\), & \text{if } c \text{ is odd,} \\
\(\frac{c}{d}\)e\(\frac{1}{24}((a+d)c-bd(c^2-1)+3d-3-3cd)\) & \text{if } c \text{ is even}.
\end{cases}
\end{equation}
For the multiplier of the theta function we have
\[
\nu_{\theta}(\gamma):= (cz+d)^{-\frac{1}{2}}\frac{\theta(\gamma z)}{\theta(z)}=\(\frac{c}{d}\)\epsilon_{d}^{-1}, \ \ \ \ \ \gamma= \left(\begin{matrix}a & b \\c & d\end{matrix}\right) \in \Gamma_0(4),
\]
where
\[
\epsilon_{d}=
\begin{cases}
1, & \text{if } d \equiv 1 \pmod{4}, \\
i, & \text{if } d \equiv 3 \pmod{4}.
\end{cases}
\]

In the next several paragraphs, we follow the exposition in \cite{Scarcity}.
If $f \in M_{k}(N,\chi\nu_{\eta}^{r})$, then $\eta^{-r}f \in M^{!}_{k-\frac{r}{2}}(N,\chi)$.
This implies that $f$ has a Fourier expansion of the form

\begin{equation}\label{lemma4.2}
f = \sum_{n \equiv r (24)}a(n)q^{\frac{n}{24}}.
\end{equation} 
These facts together imply the following lemma.
\begin{lemma}\label{integerhalfinteger}
Suppose that $0 <r < 24$ is an integer with $(r,6)=1$ and that $f \in S_{k}(N, \chi\nu_{\eta}^{r})$. Then we have 
\[
\eta^{-r}f \in M_{k-\frac{r}{2}}(N,\chi).
\]
\end{lemma}

We also have
\[
M_{k}(N,\chi\nu_{\eta}^{r}) \neq \{0\} \ \ \ \ \text{ only if } \ \ \ \ 2k-r \equiv 1-\chi(-1) \pmod{4},
\]
so if $(r,6)=1$, then $k \not \in \mathbb{Z}$.
It follows from a computation using \eqref{etamultiplier} and the formulas

\[
e\(\frac{1-d}{8}\)=\(\frac{2}{d}\)\epsilon_{d} \ \ \text{ and } \ \ \epsilon_{d_1d_2}=\epsilon_{d_1}\epsilon_{d_2}(-1)^{\frac{d_1-1}{2}\frac{d_2-1}{2}}
\]
for odd $d$, $d_1$, and $d_2$ that we have
\begin{equation}\label{passing to Shimura's space}
f(z) \in S_{k}(N,\chi \nu_{\eta}^{r}) \Longrightarrow f(24z) \in S_{k}(576N,\chi\(\tfrac{12}{\bullet}\)\nu_{\theta}^{r}).
\end{equation}
This fact, together with the usual Shimura lift on $S_{k}(576, \(\frac{12}{\bullet}\)\nu_{\theta}^{r})$, allows us to define a Shimura lift $\Sh_{t}$ on $S_{k}(1, \nu_{\eta}^{r})$ for each squarefree integer $t$ with $(t,6)=1$. Its action on Fourier expansions is
\begin{equation}\label{ShimuraCorrespondence}
\Sh_{t}\(\sum a(n)q^{\frac{n}{24}}\)= \sum A_{t}(n)q^{n},
\end{equation}
where the coefficients $A_{t}(n)$ are given by
\begin{equation}\label{ShimuraCorrespondenceCoefficients}
A_{t}(n)= \sum_{d \mid n}\(\mfrac{-1}{d}\)^{k-\frac{1}{2}}\(\mfrac{12t}{d}\)d^{k-\frac{3}{2}}a\(\frac{tn^{2}}{d^{2}}\).
\end{equation}
The work of Shimura \cite{Shimura} and Niwa \cite{Niwa} shows that, for each squarefree integer $t$ with $(t,6)=1$, we have
\[
f \in S_{k}(1,\nu_{\eta}^{r}) \implies \Sh_{t}(f) \in S_{2k-1}(288).
\]

We make use of Hecke operators on the spaces we will consider. If $k$ is an integer, we denote the Hecke operator on $S_{k}(N)$ by $T(p,k,1)$. If $k \in \frac{1}{2}\Z\setminus \Z$ and $4 \mid N$ and $r$ is a positive integer with $(r,6)=1$, we denote the Hecke operator on $S_{k}(N,\chi\(\frac{12}{\bullet}\)\nu_{\theta}^{r})$ by $T(p^{2},k, \chi)$ (these forms are half-integral weight forms in the sense of \cite{Shimura}). 

We next recall the $U$ and $V$ operators. For a positive integer $m$, we define them on Fourier expansions by
\[
\(\sum_{n=1}^{\infty}a(n)q^{\frac{n}{24}}\)\sl U_{m}:= \sum_{n=1}^{\infty}a(mn)q^{\frac{n}{24}},
\]

\[
\(\sum_{n=1}^{\infty}a(n)q^{\frac{n}{24}}\)\sl V_{m}:= \sum_{n=1}^{\infty}a(n)q^{\frac{mn}{24}}.
\]
The following facts appear as Lemma $2.1$ in \cite{Scarcity}.

\begin{lemma}\label{UandVproperties}
Suppose that $r$ is a positive integer with $(r,6)=1$, that $f \in M_{k}(N,\chi\nu_{\eta}^{r})$, and that $m$ is a positive integer. Then
\[
f\sl U_{m}=m^{\frac{k}{2}-1}\sum_{t \ (m)}f\sk \left(\begin{matrix}1 & 24t \\0 & m\end{matrix}\right) \ \ \ \text{ and } \ \ \ f\sl V_{m}=m^{-\frac{k}{2}}f\sk \left(\begin{matrix}m & 0 \\0 & 1\end{matrix}\right).
\]
\end{lemma}
A computation using Lemma~\ref{UandVproperties} and \eqref{etamultiplier} implies that if $(r,6)=1$ and $p \geq 5$ is prime, then
\begin{equation}\label{Up}
U_p:M_{k}(N,\chi\nu_{\eta}^{r}) \rightarrow M_{k}\(N\frac{p}{(N,p)},\chi\chi_{p}\nu_{\eta}^{pr}\),
\end{equation}

\begin{equation}\label{Vp}
V_p:M_{k}(N,\chi\nu_{\eta}^{r}) \rightarrow M_{k}\(Np,\chi\chi_{p}\nu_{\eta}^{pr}\).
\end{equation}
Denote by $\chi_{p}=\(\frac{\bullet}{p}\)$ the quadratic character of modulus $p$ and by $\chi_{p}^{\text{triv}}$ the trivial character with modulus $p$. If $(r,6)=1$ and if $\psi= \chi_{p} \text{ or } \chi_{p}^{\text{triv}}$, we define the twist of $f \in M_{k}(N, \chi\nu_{\eta}^{r})$ with Fourier expansion $f= \displaystyle\sum_{n=1}^{\infty}a(n)q^{\frac{n}{24}}$ by

\begin{equation}\label{twist}
f \otimes \psi:= \sum_{n=1}^{\infty}\psi(n)a(n)q^{\frac{n}{24}}.
\end{equation}

If $p \geq 5$ is prime, the fact that $p^{2} \equiv 1 \pmod{24}$ together with \eqref{Up} and \eqref{Vp} implies that we have

\begin{equation}\label{chi TRIV}
f \in M_{k}(N,\chi\nu_{\eta}^{r}) \implies f\otimes \chi_{p}^{\text{triv}} \in M_{k}(Np^{2},\chi\nu_{\eta}^{r}).
\end{equation}
We also have
\[
f \otimes \chi_{p}=\frac{1}{\epsilon_{p}\sqrt{p}}\sum_{t \ (p)}\chi_{p}(t)f\sk \left(\begin{matrix}1 & \frac{24t}{p} \\0 & 1\end{matrix}\right).
\]
A similar computation as for \eqref{Up} and \eqref{Vp} shows that
\begin{equation}\label{chi P}
f \otimes \chi_{p} \in M_{k}(Np^{2},\chi\nu_{\eta}^{r}).
\end{equation}

We next review some facts about the algebra of modular forms modulo $\ell$ and filtrations. Suppose that $\ell \geq 5$ is prime, that $K$ is a number field, and that $v$ is a prime ideal of $K$ above $\ell$. Let $\mathcal{O}_{v}$ be the ring of $v$-integral elements of $K$, and let the residue field be $\mathbb{F}_{v}$.
For Fourier expansions $\sum a(n)q^{\frac{n}{24}} \in \mathcal{O}_{v}[[q^{\frac{1}{24}}]]$, 
we define $\bar{f}:= \sum \bar{a(n)}q^{\frac{n}{24}} \in \mathbb{F}_{v}[[q^{\frac{1}{24}}]]$, 
 and we define 

\[
\bar{M_{k}(N, \nu)}:= \{\bar{f}: \ \ f \in M_{k}(N,\nu)\},
\]

\[
\bar{S_{k}(N, \nu)}:=\{\bar{f}:\ \ f \in S_{k}(N,\nu)\}.
\]
If $k$ is an integer and $f \in M_{k}(N)$, then we define the filtration of $f$ to be 

\[
\omega(f)=\omega(\bar{f}):= \inf\{k': \ \text{there exists } g \in M_{k'}(N)\text{ with } \bar{f}=\bar{g}\}.
\]
 If $f \in M_{k}(N, \nu)$, we define
\begin{equation}\label{RamanujanTheta}
\Theta:=\frac{1}{2\pi i}\frac{d}{dz}=q\frac{d}{dq}.
\end{equation}
We make use of the following facts from \cite[~$\mathsection$$4$]{Gross}.
\begin{proposition}\label{Gross}
Suppose that $N$ is a positive integer, that $k$ is an integer, that $\ell \geq 5$ is prime with $\ell \nmid N$, and that $f \in M_{k}(N)$. Then
\begin{enumerate}
\item
There exists a form $g \in M_{k+\ell+1}(N)$ with $\bar{g}=\bar{\Theta(f)}$.
\item
$\omega(\Theta(f)) \leq \omega(f)+\ell+1$, with equality if and only if $\ell \nmid \omega(f)$.
\item
$\omega(f^{i})=i\omega(f)$ for all $i \geq 1$.
\item
If $g \in M_{k'}(N)$ has $\bar{f}=\bar{g}\neq 0$, then $k \equiv k' \pmod{\ell-1}$.
\end{enumerate}
\end{proposition}
Finally, if $k \geq 2$ is an even integer, we denote the weight $k$ Eisenstein series by $E_k$.

\section{Preliminary Results}
We record here some preliminary results which we require for the proof of Theorem~\ref{thm:main}.

\begin{lemma}\label{Lemma1}
Let $\ell \geq 5$ be a prime and $K$ is a number field. Suppose that $r$ is a positive integer with $(r,6)=1$ and that $f\in S_{\lambda+\frac{1}{2}}(1, \nu^{r}_{\eta})$. Then there is a form $F \in S_{\lambda+\ell+1+\frac{1}{2}}(1, \nu^{r}_{\eta})$
such that $F \equiv \Theta (f) \pmod{\ell}$.
\end{lemma}

\begin{proof}
Choose an integer $j$ such that $\ell j+r \equiv 0 \pmod{24}$. Define

\[
g:= \eta^{\ell j}f \in S_{\lambda+\frac{\ell j}{2}+\frac{1}{2}}(1).
\]
There exists a form $G \in S_{\lambda+\frac{\ell j}{2}+\ell+1+\frac{1}{2}}(1)$ such that

\[
G \equiv \Theta (g) \equiv \eta^{\ell j}\Theta (f) \pmod{\ell}.
\]
Let $d=\dim(S_{\lambda+\frac{\ell j}{2}+\ell+1+\frac{1}{2}}(1))$. The space $S_{\lambda+\frac{\ell j}{2}+\ell+1+\frac{1}{2}}(1)$ has a basis $\{f_1,...,f_d\}$ with integer coefficients with the property 
\begin{equation}\label{eqn:INTEGRALITY}
f_i(z)=q^{i}+O(q^{i+1}),  \ \ \ \ \ \ \ \ \ \ \ \ \ \ \ \ \ 1 \leq i \leq d.
\end{equation}
After subtracting an appropriate linear combination of these basis elements from $G$, we may assume that $F:= \frac{G}{\eta^{\ell j}} \in S_{\lambda+\ell+1+\frac{1}{2}}(1,\nu^{r}_{\eta})$. We have

\[
F \equiv \Theta (f) \pmod{\ell}.
\]

\end{proof}

\begin{lemma}\label{Lemma2}
Suppose that $\ell \geq 5$ is prime, that $K$ is a number field which is Galois over $\mathbb{Q}$, and that $v$ is a prime of $K$ above $\ell$. Suppose that $\lambda$ is a non-negative integer, that $r$ is a positive integer with $(r,6)=1$, and that $g \in S_{\lambda+\frac{1}{2}}(1, \nu_{\eta}^{r})$ satisfies

\[
g \equiv \sum_{n=1}^{\infty}a(n)q^{\frac{\ell n}{24}} \not \equiv 0 \pmod{v}.
\]
Then there exists $\lambda' \geq 0$ with $\lambda'+\frac{1}{2} \leq \frac{1}{\ell}(\lambda+\frac{1}{2})$, and a form $f \in S_{\lambda'+\frac{1}{2}}(1, \nu_{\eta}^{r\ell})$ such that
\[
f \equiv  g\sl U_{\ell} \equiv  \sum_{n=1}^{\infty}a(n)q^{\frac{n}{24}}  \pmod v.
\]

\end{lemma}

\begin{proof}
Choose an integer $j>0$ such that $\ell j+r \equiv 0 \pmod{24}$, and define 
\[
h:= \eta^{\ell j}g \in S_{\lambda+\frac{\ell j}{2}+\frac{1}{2}}(1).
\]
Suppose that $x \in \mathcal{O}_{v}$ and  that $\sigma \in \Gal(K/\mathbb{Q})$ is a Frobenius automorphism for the prime $v$. Then we have $x^{\sigma} \in \mathcal{O}_{v}$ and 

\[
x^{\sigma} \equiv x^{\ell} \pmod{v}.
\]
Note that $\sigma$ preserves the space $S_{\lambda+\frac{\ell j}{2}+\frac{1}{2}}(1)$. Since $U_{\ell}$ acts as  
$T(\ell, \lambda+\frac{\ell j}{2}+\frac{1}{2},1)$ modulo $v$, we see that $\bar{h\sl U_{\ell}} \in \bar{S_{\lambda+\frac{\ell j}{2}+\frac{1}{2}}(1)}$. We have
\[
\bar{h^{\sigma}}=(\bar{h\sl U_{\ell}})^{\ell}.
\]
By $(4)$ of Proposition~\ref{Gross}, we know that there exists an integer $\beta \geq 0$ such that

\[
k:= \omega(\bar{h\sl U_{\ell}})=\frac{1}{\ell}\omega(\bar{h^{\sigma}})=\frac{1}{\ell}\(\lambda-\beta(\ell-1)+\frac{\ell j}{2}+\frac{1}{2}\).
\]
Therefore, arguing as in the proof of Lemma~\ref{Lemma1}, we can find a form
$H \in S_{k}(1)$ such that $\bar{H}=\bar{h\sl U_{\ell}}=\bar{\eta^{j}(g\sl U_{\ell}})$ and
$f:= \frac{H}{\eta^{j}} \in S_{\lambda+\ell+1+\frac{1}{2}}(1)$. Then, we see that $f \in S_{k-\frac{j}{2}}(1,\nu_{\eta}^{r\ell})$, and we have $\bar{f}=\bar{g\sl U_{\ell}}$.  
The lemma follows since $k-\frac{j}{2} \leq \frac{1}{\ell}(\lambda+\frac{1}{2})$.

\end{proof}

\section{Preliminary Reduction}
Before proving Theorem~\ref{thm:main}, we reduce the number of square classes on which our forms may be supported and the number of multipliers which we must consider.

\begin{proposition}\label{Proposition1}
Suppose that $\ell \geq 5$ is prime, that $K$ is a number field, and that $v$ is a prime above $\ell$. Suppose that $\lambda$ is a non-negative integer, that $r$ is a positive integer with $(r,6)=1$, and that $f \in S_{\lambda+\frac{1}{2}}(1, \nu_{\eta}^{r})$. Further, suppose that

\begin{equation}\label{finitelymanysquareclasses}
 f \equiv \sum_{i=1}^{m}\sum_{n=1}^{\infty}a(t_{i}n^2)q^{\frac{t_{i}n^{2}}{24}}  \not \equiv 0 \pmod v,
 \end{equation}
  where each $t_{i}$ is a square-free positive integer. Then
 \begin{equation}\label{twosquareclasses}
 f \equiv \sum_{n=1}^{\infty}a(n^2)q^{\frac{n^{2}}{24}}+\sum_{n=1}^{\infty}a(\ell n^{2})q^{\frac{\ell n^{2}}{24}} 
 \pmod v.
 \end{equation}
\end{proposition}

\begin{proof}
Fix an $i \in \{1,...,m \}$. We may assume that there exists an integer $n_{i}$ for which $a(t_{i}n_{i}^{2}) \not \equiv 0 \pmod{v}$. Recalling our notation \eqref{twist} and the facts \eqref{chi TRIV} and \eqref{chi P}, we follow the argument in the proof of Lemma $4.1$ of \cite{Ahlgren-Boylan} to find primes $p_{1},...,p_{n} \geq 5$, each relatively prime to $n_{i}t_{i}\ell$ and a form
\[
G_{i} \in S_{\lambda+\frac{1}{2}}(p_{1}^{2}\cdots p_{n}^{2},\nu_{\eta}^{r})
\]
satisfying
\[
G_{i} \equiv \sum_{(n,\prod p_{j})=1}a(t_{i}n^{2})q^{\frac{t_{i}n^{2}}{24}} \not \equiv 0 \pmod{v}.
\]
Note that 

\[
G_{i}^{24} \in S_{24\lambda+12}(p_{1}^{2},...,p_{s}^{2}).
\]
Since 
\[
G_{i}^{24} \equiv \sum_{n=1}^{\infty}b(t_{i}n)q^{t_{i}n} \pmod{v}
\]
for some coefficients $b(t_{i}n)$, we can apply 
the following result
to conclude that $t_{i}=1 \text{ or } \ell$.

\begin{theorem}{\cite[Thm 3.1]{Ahlgren-Choi-Rouse}}\label{thm:ACR THM 3.1}
Suppose that $K$ is a number field and that $v$ is a prime above $\ell$ with ring of $v$-integral elements $\mathcal{O}_{v}$. Suppose that $k$ is positive integer and that 
\[
f=\sum_{n=1}^{\infty}a(n)q^{n} \in S_{2k}(\Gamma_0(N)). 
\]
If $t > 1$ satisfies $(t,\ell N)=1$ and
\[
f \equiv \displaystyle \sum_{n=1}^{\infty}a(tn)q^{tn} \pmod{v},
\]
then $f \equiv 0 \pmod v$.
\end{theorem}

\end{proof}
The next result reduces the number of multipliers which we must consider.

\begin{lemma}\label{lemmatwomultipliers}

Suppose that $r$ is a positive integer with $(r,6)=1$ and that $f \in S_{\lambda+\frac{1}{2}}(1, \nu_{\eta}^{r})$ satisfies \eqref{twosquareclasses}. Then we have
\begin{equation}\label{twomultipliers}
r \equiv 1 \pmod{24} \ \ \ \ \text{ or } \ \ \ \ r \equiv \ell \pmod{24}.
\end{equation}
\end{lemma}

\begin{proof}
Since $f$ satisfies \eqref{twosquareclasses}, it follows that either $a(n^{2}) \neq 0$ or $a(\ell n^{2}) \neq 0$ for some positive integer $n$. It follows from \eqref{lemma4.2} and the fact that $r^{2} \equiv 1 \pmod{24}$ whenever $(r,6)=1$ that we have \eqref{twomultipliers}. 
\end{proof}

\section{Proof of Theorem~\ref{thm:main}}
The proof of Theorem~\ref{thm:main} will proceed in several steps. We first consider the case when $r \equiv 1 \pmod{24}$ and $\lambda$ is even.

\begin{theorem}\label{Theorem1}
Suppose that $\ell \geq 5$ is prime, that $K$ is a number field, and that $v$ is a prime above $\ell$. Suppose that $\lambda$ is a non-negative integer and that $f \in S_{\lambda+\frac{1}{2}}(1, \nu_{\eta})$ has the form \eqref{twosquareclasses}. If $\lambda$ is even and $\lambda < 2\ell^{2}+\ell-1$, then

\[
\sum_{\ell \nmid n}a(n^{2})q^{\frac{n^{2}}{24}} \equiv a(1)\sum_{\ell \nmid n}\(\frac{12}{n}\)n^{\lambda}q^{\frac{n^{2}}{24}} \pmod{v}.
\]

\end{theorem}

\begin{proof}
Define $\bar{\lambda}:= \lambda \pmod{\ell-1}$. By Lemma~\ref{Lemma1}, we have forms $g(z) \in S_{\lambda+\ell+1+\frac{1}{2}}(1,\nu_{\eta})$ and 
$h(z) \in S_{(\ell+1)\frac{\bar{\lambda}+2}{2}+\frac{1}{2}}(1,\nu_{\eta})$ such that

\[
g=\sum_{n=1}^{\infty}c(n)q^{\frac{n}{24}} \equiv \Theta(f) \equiv \sum_{n=1}^{\infty}\frac{n^2}{24}a(n^2)q^{\frac{n^2}{24}} \pmod{v}
\]
and
\[
h=\sum_{n=1}^{\infty}b(n)q^{\frac{n}{24}} \equiv 24^{\frac{\bar{\lambda}}{2}}a(1)\Theta^{\frac{\bar{\lambda}+2}{2}}(\eta) \equiv a(1)\sum_{n=1}^{\infty}\(\frac{12}{n}\)\frac{n^{\bar{\lambda}+2}}{24}q^{\frac{n^2}{24}} \pmod{v}.
\]

It suffices to show that $g \equiv h \pmod{v}$.  To this end, we make use of this theorem, which follows from an argument of Bruinier and Ono \cite[Thm 3.1]{Bruinier-Ono} (see \cite[Thm 2.1]{Ahlgren-Choi-Rouse}).

\begin{theorem}\label{Bruinier-Ono}
Suppose that $N$ is a positive integer with $4 \mid N$. Suppose that $\ell \geq 5$ is prime, that $K$ is a number field, and that $v$ is a prime of $K$ above $\ell$. 
Suppose that $\lambda$ is a non-negative integer and that $r$ is a positive integer with $(r,6)=1$. Suppose that

\[
f(z)=\sum_{n=1}^{\infty}a(n)q^{n} \in S_{\lambda+\frac{1}{2}}(N,\chi\(\tfrac{12}{\bullet}\) \nu_{\theta}^{r}),
\]
that $\ell \nmid N$, and that $p \nmid N\ell$ is prime. If there exists $\epsilon_{p} \in \{ \pm 1\}$ such that

\[
f(z) \equiv \sum_{\( \frac{n}{p}\) \in \{0, \epsilon_{p}\}}a(n)q^{n} \pmod v,
\]
then we have
\[
(p-1)f(z)\sl T\(p^2,\lambda +\mfrac{1}{2},\chi\) \equiv \epsilon_{p}\chi(p)\(\tfrac{(-1)^{\lambda}}{p}\)(p^{\lambda}+p^{\lambda-1})(p-1)f(z) \pmod v.
\]
\end{theorem}

By \eqref{passing to Shimura's space}, we can apply Theorem~\ref{Bruinier-Ono} to $g(24z)$ to conclude that, for odd primes $p \geq 5$ with $p \not \equiv 0,1 \pmod{\ell}$, we have

\begin{equation}\label{applicationofBruinierOno}
g(24z) \sl T\(p^2,\lambda+\ell+1+\mfrac{1}{2},1\) \equiv \(\frac{12}{p}\)(p^{\bar{\lambda}+2}+p^{\bar{\lambda}+1})g(24z) \pmod{v}.
\end{equation}
Suppose that $n$ is a positive integer satisfying $(n,6)=1$ which is divisible only by primes $p \not \equiv 0,1 \pmod{\ell}$.
 If $p$ is such a prime, write $n=p^{a}n_{0}$ if $p^{a} \mid \mid n$.
The definition of the Hecke operator on $S_{\lambda+\ell+1+\frac{1}{2}}(576,(\frac{12}{\bullet})\nu_{\theta}^{r})$ implies that we have
\[
c(n^2p^2)+p^{\bar{\lambda}+1}\(\frac{12n^2}{p}\)c(n^2)+p^{2\bar{\lambda}+3}c\(\frac{n^2}{p^2}\) \equiv \(\frac{12}{p}\)(p^{\bar{\lambda}+2}+p^{\bar{\lambda}+1})c(n^2) \pmod{v},
\]
and an induction argument on $a$ then implies that
\[
c(p^{2a}n_{0}^{2}) \equiv \(\frac{12}{p}\)^{a}p^{a(\bar{\lambda}+2)}c(n_{0}^{2}) \pmod{v}.
\]

Thus, we have
\[
c(n^{2}) \equiv \(\frac{12}{n}\)n^{\bar{\lambda}+2}c(1) \equiv \(\frac{12}{n}\)\frac{n^{\bar{\lambda}+2}}{24}a(1) \equiv b(n^2) \pmod{v}.
\]
This shows that the coefficients $c(n^{2})$ and $b(n^{2})$ agree whenever $n$ is a positive integer such that $(n,6)=1$ which is divisible only by primes $p \geq 5$ with $p \not \equiv 0,1 \pmod{\ell}$.

Now define
\[
k:= \max\{\lambda+\ell+1,(\ell+1)\frac{\bar{\lambda}+2}{2}\}. 
\]
These numbers agree modulo $\ell-1$ by virtue of $\lambda$ being even, so by multiplying $g$ or $h$ by an appropriate power of $E_{\ell-1} \equiv 1 \pmod{\ell}$, we see that there exist forms $g_{1}$ and $h_{1}$ in $S_{k+\frac{1}{2}}(1,\nu_{\eta})$ such that $g_{1} \equiv g \pmod{v}$ and $h_{1} \equiv h \pmod{v}$. 
Thus, to prove the theorem, it suffices to show that $g_{1} \equiv h_{1} \pmod{v}$.
Note that $c(n^2) \equiv b(n^2) \equiv 0 \pmod{v}$ for positive integers $n$ such that $(n,6) \neq 1$ by \eqref{lemma4.2}, and that $c(n^2)$ and $b(n^2)$ vanish modulo $\ell$ whenever $n$ is divisible by $\ell$. This implies that $c(n) \equiv b(n) \pmod{v}$ whenever $n < (2\ell+1)^{2}$.
 Thus,
$\eta^{-1}(g_{1}-h_{1}) \in M_{k}(1)$ is of the form

\[
\eta^{-1}(g_{1}-h_{1}) \equiv cq^{\frac{\ell^{2}+\ell}{6}}+ \cdots \pmod{v}
\]
for some $c \in \mathcal{O}_{v}$. By arguing as in the proof of Lemma~\ref{Lemma1}, we may assume that $\eta^{-1}(g_{1}-h_{1}) \in S_{k}(1)$. To prove that $g_{1} \equiv h_{1} \pmod{v}$, it suffices to show by \cite[Thm 1]{Sturm} that 

\[
\frac{k}{12} < \frac{\ell^{2}+\ell}{6}.
\]
Since $\bar{\lambda} < \ell$, we have

\[
\frac{(\ell+1)(\bar{\lambda}+2)}{24}<\frac{\ell^{2}+\ell}{6}.
\]
Since $\lambda < 2\ell^{2}+\ell-1$, we have
\[
\frac{\lambda+\ell+1}{12}<\frac{\ell^{2}+\ell}{6}.
\]
The result follows.
\end{proof}

We now consider what happens when $\lambda$ is odd.

\begin{proposition}\label{oddcases}
Suppose that $\ell \geq 5$ is prime, that $K$ is a number field, and that $v$ is a prime of $K$ above $\ell$. Suppose that $\lambda$ is a non-negative integer, that $r$ is a positive integer with $(r,6)=1$, and that $f \in S_{\lambda+\frac{1}{2}}(1, \nu_{\eta}^{r})$ has the form
\[
f \equiv \sum_{n=1}^{\infty}a(n^{2})q^{\frac{n^{2}}{24}}+\sum_{n=1}^{\infty}a(\ell n^{2})q^{\frac{\ell n^{2}}{24}} \not \equiv 0 \pmod{v}.
\]
If $\lambda$ is odd, then $\Theta(f) \equiv 0 \pmod{v}$.
\end{proposition}

\begin{proof}
Suppose by way of contradiction that $\Theta(f) \not \equiv 0 \pmod{v}$. By Lemma~\ref{Lemma1}, there exists $g \in S_{\lambda+\ell+1+\frac{1}{2}}(1,\nu_{\eta}^{r})$ such that

\[
g \equiv \sum_{n=1}^{\infty} \frac{n^{2}}{24}a(n^{2})q^{\frac{n^{2}}{24}} \not \equiv 0 \pmod{v},
\]
so there exists $n_{0}$ such that $a(n_{0}^{2}) \neq 0$.
By \eqref{lemma4.2}, we have $r \equiv 1 \pmod{24}$. By Lemma~\ref{integerhalfinteger}, we then have $\eta^{-1}f \in M_{\lambda}(1)=\{0\},$ which is a contradiction. 
Thus, $\Theta(f) \equiv 0 \pmod{v}$.
\end{proof}

We require one more result before proving Theorem~\ref{thm:main}.

\begin{proposition}\label{Spicy}
Suppose that $\ell \geq 5$ is prime, that $K$ is a number field which is Galois over $\Q$, and that $v$ is a prime of $K$ above $\ell$. 
Suppose that $r$ is a positive integer with $(r,6)=1$ and that $g \in S_{\lambda'+\frac{1}{2}}(1, \nu_{\eta}^{r})$ satisfies 

\[
 g \equiv \sum_{n=1}^{\infty}a(n^2)q^{\frac{n^{2}}{24}}+\sum_{n=1}^{\infty}a(\ell n^{2})q^{\frac{\ell n^{2}}{24}} \not \equiv 0 \pmod v. 
 \]
If $\lambda' < \frac{\ell-1}{2}$, then $\lambda'=0$, $r=1$, and $g=c\eta$ for 
some $c \in \mathcal{O}_{v}$. 
\end{proposition}

\begin{proof}
First assume that $\lambda'=0$. Assume without loss of generality that $0<r <24$. By Lemma~\ref{integerhalfinteger}, we have $\eta^{-r}g \in M_{\frac{1}{2}-\frac{r}{2}}(1)$, from which $r=1$ follows. Thus, $\eta^{-1}g \in M_0(1) = \mathcal{O}_{v}$, and the result follows.

Now assume that $1 \leq \lambda' < \frac{\ell-1}{2}$.
Suppose first that $\lambda'$ is odd. Proposition~\ref{oddcases} implies that $\Theta (g) \equiv 0 \pmod{v}$, so we have 
\[
g \equiv \sum_{n=1}^{\infty}a(\ell n^{2})q^{\frac{\ell n^{2}}{24}}+\sum_{n=1}^{\infty}a(\ell^{2}n^{2})q^{\frac{\ell^{2}n^{2}}{24}} \pmod{v}.
\]
By Lemma~\ref{Lemma2}, there exists $f \in S_{\lambda^{*}+\frac{1}{2}}(1,\nu_{\eta}^{r\ell})$ with 
\[
f \equiv \sum_{n=1}^{\infty}a(\ell n^{2})q^{\frac{n^{2}}{24}}+\sum_{n=1}^{\infty}a(\ell^{2}n^{2})q^{\frac{\ell n^{2}}{24}} \pmod{v},
\]
and $\lambda^{*}+\frac{1}{2} \leq \frac{1}{\ell}(\lambda^{'}+\frac{1}{2})$. Since $\lambda^{*} \geq 0$, we have $\frac{1}{2} \leq \frac{1}{\ell}(\lambda^{'}+\frac{1}{2})$, which would imply that $\lambda' \geq \frac{\ell-1}{2}$, so $\lambda'$ cannot be odd.

Now assume that $2 \leq \lambda' < \frac{\ell-1}{2}$  and $\lambda'$ is even. Applying Theorem~\ref{Theorem1}, we see that 

\[
\Theta^{\ell-1}(g) \equiv a(1) \sum_{\ell \nmid n}\(\frac{12}{n}\)n^{\lambda'}q^{\frac{n^{2}}{24}} \pmod{v}, 
\]
which implies that
\begin{equation}\label{somethingkewl}
a(n^{2}) \equiv a(1) \(\frac{12}{n}\)n^{\lambda'} \pmod{v} \ \ \ \text{ if } \ \ \ \ell \nmid n.
\end{equation}

We show that $a(1) \equiv 0 \pmod{\ell}$. Assume to the contrary that $a(1) \not \equiv 0 \pmod{\ell}$. Consider the Shimura lift

\begin{equation}
G:= \Sh_{1}(g)=\sum_{n=1}^{\infty}b(n)q^{n} \in S_{2\lambda'}(288)
\end{equation}
from \eqref{ShimuraCorrespondence}. 
If $\ell \nmid n$, then \eqref{ShimuraCorrespondenceCoefficients} and \eqref{somethingkewl} give
\begin{equation}\label{*}
b(n) \equiv a(1)\(\frac{12}{n}\)\sum_{d \mid n}d^{\lambda'-1}\(\frac{n}{d}\)^{\lambda'}=a(1)\(\frac{12}{n}\)n^{\lambda'-1}\sigma_{1}(n).
\end{equation}
This is congruent to the $n^{\text{th}}$ coefficient of 
\[
-\frac{a(1)}{24}\Theta^{\lambda'-1}\(E_{\ell+1}\otimes \(\frac{12}{\bullet}\)\) \equiv -\frac{a(1)}{24}\Theta^{\lambda'-1}\(E_{2}\otimes\(\frac{12}{\bullet}\)\) \pmod{\ell}.
\]

Note that $E_{2}\otimes(\frac{12}{\bullet}) = (E_{2}-2E_{2}\sl V_{2})\otimes\(\frac{12}{\bullet}\)$. 
This implies that the filtration of $E_{2}\otimes (\frac{12}{\bullet})$ on level $288$ is $2$, since $E_{2}-2E_{2}\sl V_{2} \in M_{2}(2)$ \cite[~$\mathsection$$1.2$]{DiamondShurman}.
By $(2)$ of Proposition~\ref{Gross}, we have
\[
\omega\(\bar{\Theta^{\lambda'-1}\(E_{\ell+1}\otimes\(\frac{12}{\bullet}\)\)}\)=\ell(\lambda'-1)+\lambda'+1.
\]
We also have
\[
G \equiv GE_{\ell-1}^{\lambda'-1} \in S_{\ell(\lambda'-1)+\lambda'+1}(288).
\]
By $(1)$ of Proposition~\ref{Gross}, there exists $H \in S_{\ell(\lambda'-1)+\lambda'+1}(288)$ with
\[
\bar{H}= \bar{GE_{\ell-1}^{\lambda'-1}+\frac{a(1)}{24}\Theta^{\lambda'-1}\(E_{\ell+1}\otimes\(\frac{12}{\bullet}\)\)} \in \bar{S_{\ell(\lambda'-1)+\lambda'+1}(288)}.
\]
Let $\sigma \in \Gal(K/\mathbb{Q})$ be a Frobenius element for $v$. Note that $\bar{H}$ has the form $\sum \bar{a(n)}q^{\ell n}$, so we have
\[
\bar{H}^{\sigma}=(\bar{H\sl U_{\ell}})^{\ell}. 
\]
By $(3)$ of Proposition~\ref{Gross}, we have
\[
\omega(\bar{H})=\omega(\bar{H}^{\sigma})=\ell\omega(\bar{H\sl U_{\ell}}).
\]
If $a(1) \not \equiv 0 \pmod{\ell}$, then $\omega(\bar{\frac{a(1)}{24}\Theta^{\lambda'-1}\(E_{\ell+1}\otimes\(\frac{12}{\bullet}\)\)}) > \omega(\bar{GE_{\ell-1}^{\lambda'-1}})$ since $\omega(GE_{\ell-1}^{\lambda'-1}) \leq 2\lambda'$ and $\lambda'$ is even. This would imply that
$\omega(\bar{H})=\omega(\bar{\Theta^{\lambda'-1}\(E_{\ell+1}\otimes\(\frac{12}{\bullet}\)\)})= \ell(\lambda'-1)+\lambda'+1$. This contradicts the fact that $\omega(\bar{H})$ is a multiple of $\ell$.
Thus, $a(1) \equiv 0 \pmod{\ell}$. By \eqref{*}, we have $\Theta(g) \equiv 0 \pmod{v}$. The result now follows as in the odd case.

\end{proof}

Now we prove Theorem~\ref{thm:main}.
\begin{proof}[Proof of Theorem~\ref{thm:main}]
Suppose that $\ell \geq 5$ is prime, that $K$ is a number field, and that $v$ is a prime of $K$ above $\ell$.
We may assume that $K$ is Galois over $\mathbb{Q}$. 
Suppose that $r$ is a positive integer satisfying $(r,6)=1$, that $\lambda$ is a non-negative integer satisfying $\lambda+\frac{1}{2} < \frac{\ell^{2}}{2}$, and that $f \in S_{\lambda+\frac{1}{2}}(1, \nu_{\eta}^{r})$ has the property that

\[
f \equiv \sum_{i=1}^{m}\sum_{n=1}^{\infty}a(t_{i}n^{2})q^{\frac{t_{i}n^{2}}{24}} \not \equiv 0 \pmod{v}.
\]
By Proposition~\ref{Proposition1} and Lemma~\ref{lemmatwomultipliers}, we may assume that

\[
f \equiv \sum_{n=1}^{\infty}a(n^{2})q^{\frac{n^{2}}{24}}+\sum_{n=1}^{\infty}a(\ell n^{2})q^{\frac{\ell n^{2}}{24}} \not \equiv 0 \pmod{v}
\]
and that either $r \equiv 1 \pmod{24}$ or $r \equiv \ell \pmod{24}$. So, we need only consider the cases when $f \in S_{\lambda+\frac{1}{2}}(1, \nu_{\eta})$ and when $f \in S_{\lambda+\frac{1}{2}}(1,\nu_{\eta}^{\ell})$ with $\ell \not \equiv 1 \pmod{24}$.

Suppose that $f \in S_{\lambda+\frac{1}{2}}(1,\nu_{\eta})$. Assume that $\lambda$ is even. If $\lambda=0$, then $f=c\eta$ for some $c \in \mathcal{O}_{v}$. This implies that
\[
f = a(1)\sum_{n=1}^{\infty}\(\frac{12}{n}\)q^{\frac{n^{2}}{24}},
\] 
which has the form of case $(1)$ of Theorem~\ref{thm:main}, so assume that $\lambda > 0$.
Theorem~\ref{Theorem1} then implies that
\begin{equation}\label{**}
\Theta^{\ell-1}(f)=\sum_{\ell \nmid n}a(n^{2})q^{\frac{n^{2}}{24}} \equiv a(1) \sum_{\ell \nmid n}\(\frac{12}{n}\)n^{\lambda}q^{\frac{n^{2}}{24}} \pmod{v}.
\end{equation}
Define $\bar{\lambda}:=\lambda \pmod{\ell-1}.$ By Lemma~\ref{Lemma1}, we have
\begin{equation}\label{R.L.C.}
\bar{\Theta^{\ell-1}(f)}=\bar{24^{\frac{\bar{\lambda}}{2}}a(1)\Theta^{\frac{\bar{\lambda}}{2}}(\eta)} \in \bar{S_{\frac{\bar{\lambda}}{2}(\ell+1)+\frac{1}{2}}(1,\nu_{\eta})}.
\end{equation}
The fact that
\[
\(\frac{\bar{\lambda}}{2}\)(\ell+1)+\frac{1}{2} < \frac{\ell^{2}}{2}
\]
implies that $\bar{f-\Theta^{\ell-1}(f)} \in \bar{S_{\lambda'+\frac{1}{2}}(1,\nu_{\eta})}$, where $\lambda'+\frac{1}{2} < \frac{\ell^{2}}{2}$. Since
\[
f-\Theta^{\ell-1}(f) \equiv \sum_{n=1}^{\infty}a(\ell n^{2})q^{\frac{\ell n^{2}}{24}}+\sum_{n=1}^{\infty}a(\ell^{2}n^{2})q^{\frac{\ell^{2} n^{2}}{24}} \pmod{v},
\]
we apply Lemma~\ref{Lemma2} to conclude that there exists $g \in S_{\lambda^{*}+\frac{1}{2}}(1,\nu_{\eta}^{\ell})$ 
with $\lambda^{*}+\frac{1}{2} < \frac{\ell}{2}$ satisfying
\[
g \equiv \(f-\Theta^{\ell-1}(f)\)\sl U_{\ell} \equiv \sum_{n=1}^{\infty}a(\ell n^{2})q^{\frac{n^{2}}{24}}+\sum_{n=1}^{\infty}a(\ell^{2} n^{2})q^{\frac{\ell n^{2}}{24}} \pmod{v}.
\]
If $g \equiv 0 \pmod{v}$, then
\[
f \equiv \Theta^{\ell-1}(f)\pmod{v},
\]
and, by \eqref{**}, this proves that
\[
f \equiv a(1) \sum_{n=1}^{\infty}\(\frac{12}{n}\)n^{\lambda}q^{\frac{n^{2}}{24}} \pmod{v}.
\]
This has the form of case $(1)$ of Theorem~\ref{thm:main}.

If $g \not \equiv 0 \pmod{v}$, then Proposition~\ref{Spicy} implies that $\lambda'=0$ and $g=c\eta$ for some $c \in \mathcal{O}_{v}$, which means that

\[
f-\Theta^{\ell-1}(f) \equiv a(\ell)\sum_{n=1}^{\infty}\(\frac{12}{n}\)q^{\frac{\ell n^{2}}{24}} \pmod{v}.
\]
Thus, we have

\begin{equation}\label{deletthis}
f \equiv a(1)\sum_{n=1}^{\infty}\(\frac{12}{n}\)n^{\lambda}q^{\frac{n^{2}}{24}}+a(\ell)\sum_{n=1}^{\infty}\(\frac{12}{n}\)q^{\frac{\ell n^{2}}{24}} \pmod{v}.
\end{equation}

Since $g \not \equiv 0 \pmod{v}$, we have $a(\ell) \not \equiv 0 \pmod{v}$. Proposition~\ref{Spicy} applied to $f$ implies that $\ell \equiv 1 \pmod{24}$. 
If $a(1) \equiv 0 \pmod{v}$, then \eqref{deletthis} has the form of case $(2)$ of Theorem~\ref{thm:main}.
This is equivalent to the congruence

\[
f \equiv a(\ell)\eta^{\ell} \pmod{v}.
\]
By $(4)$ of Proposition~\ref{Gross}, we have $\lambda \equiv \omega(\bar{\eta^{-1}f})  \equiv \omega(\bar{\eta^{\ell-1}}) \equiv \frac{\ell-1}{2}\pmod{\ell-1}$.

If $a(1) \not \equiv 0 \pmod{v}$, then \eqref{deletthis} has the form of case $(3)$ of Theorem~\ref{thm:main}. This is equivalent to the congruence

\[
f \equiv 24^{\frac{\lambda}{2}}a(1)\Theta^{\frac{\lambda}{2}}(\eta)+a(\ell)\eta^{\ell} \pmod{v}.
\]
By $(4)$ of Proposition~\ref{Gross}, we have $\omega(\bar{\eta^{-1}f}) \equiv \omega(\bar{\eta^{-1}\Theta^{\frac{\lambda}{2}}(\eta)}) \equiv \lambda \pmod{\ell-1}$. This implies that
$\omega(\bar{\eta^{\ell-1}}) \equiv \lambda \pmod{\ell-1}$. Since $\omega(\bar{\eta^{\ell-1}})=\frac{\ell-1}{2}$, we have $\lambda \equiv \frac{\ell-1}{2} \pmod{\ell-1}$.

 Now assume that $f \in S_{\lambda+\frac{1}{2}}(1,\nu_{\eta})$ and that $\lambda$ is odd. By Proposition~\ref{oddcases}, we have
 \[
 f \equiv \sum a(\ell n^{2})q^{\frac{\ell n^{2}}{24}} \pmod{v}.
 \] 
By Lemma~\ref{Lemma2} (since $\lambda+\frac{1}{2} < \frac{\ell^{2}}{2}$), there exists $g \in S_{\lambda'+\frac{1}{2}}(1, \nu_{\eta}^{\ell})$ such that $g \equiv f\sl U_{\ell} \pmod{v}$ and $\lambda'+\frac{1}{2} < \frac{\ell}{2}$. 
Proposition~\ref{Spicy} implies that $\lambda'=0$ and $g =c\eta$ for some $c \in \mathcal{O}_{v}$. Thus,
\[
f \equiv a(\ell)\sum_{n=1}^{\infty}\(\frac{12}{n}\)q^{\frac{\ell n^{2}}{24}} \pmod{v}.
\]
Since $f \not \equiv 0 \pmod{v}$ has a Fourier expansion of the form $\eqref{lemma4.2}$, we have 
$r \equiv \ell \equiv 1 \pmod{24}$
in this case. As above, this implies that $\lambda \equiv \frac{\ell-1}{2} \pmod{\ell-1}$, which implies that $\lambda$ is even. This is a contradiction, so $\lambda$ cannot be odd in this case.

Finally, suppose that $\ell \not \equiv 1 \pmod{24}$ and that $f \in S_{\lambda+\frac{1}{2}}(1,\nu_{\eta}^{\ell})$. 
By \eqref{lemma4.2}, we have
\[
f \equiv \sum_{n=1}^{\infty}a(\ell n^{2})q^{\frac{\ell n^{2}}{24}} \pmod{v}.
\]
By Lemma~\ref{Lemma2}, there is $g \in S_{\lambda'+\frac{1}{2}}(1,\nu_{\eta})$ with $\lambda'+\frac{1}{2}< \frac{\ell}{2}$ such that $g \equiv f\sl U_{\ell} \pmod{v}$.
Proposition~\ref{Spicy} implies that $\lambda'=0$ and that $g=c\eta$ for some $c \in \mathcal{O}_{v}$. Thus,
\[
f \equiv a(\ell)\sum_{n=1}^{\infty}\(\frac{12}{n}\)q^{\frac{\ell n^{2}}{24}} \pmod{v}.
\]
This has the form of case $(2)$ of Theorem~\ref{thm:main}. As above, we have $\lambda \equiv \frac{\ell-1}{2} \pmod{\ell-1}$.
\end{proof}

 \section{Acknowledgements}
The author would like to thank Scott Ahlgren for suggesting this project and for advice and guidance for this work. The author would also like to thank the referee for carefully reading this manuscript and making helpful comments which improved its exposition. Finally, the author would like to thank the Graduate College Fellowship program at the University of Illinois at Urbana-Champaign and the Alfred P. Sloan Foundation for their generous research support.

 \bibliographystyle{amsalpha}

 \bibliography{CongruencesforLevel1cuspformswithetamultiplier}
 
 \end{document}